\documentclass{article}

\usepackage{amsmath,amssymb,amsthm}
\usepackage[english]{babel}
\usepackage{enumitem}
\usepackage[top=2.5cm,bottom=2.5cm,left=2.5cm,right=2.5cm]{geometry}
\usepackage{mathrsfs}
\usepackage{tikz}
\usepackage{url}
\usepackage{xcolor}


\newtheorem{thm}{Theorem}[section]

\newtheorem{lm}[thm]{Lemma}
\newtheorem{res}[thm]{Result}
\newtheorem{crl}[thm]{Corollary}

\theoremstyle{definition}

\newtheorem{df}[thm]{Definition}

\newcommand{\NN}{\mathbb{N}}
\newcommand{\NNnot}{\NN\setminus\{0\}}
\newcommand{\FF}{\mathbb{F}}

\newcommand{\pg}{\textnormal{PG}}
\newcommand{\ag}{\textnormal{AG}}
\newcommand{\vspan}[1]{\left \langle #1 \right \rangle}

\newcommand{\FieldRed}{\mathcal{F}}
\newcommand{\Regulus}{\mathcal{R}}
\newcommand{\TransversalSpace}{\mathcal{T}}
\newcommand{\Sub}{\mathcal{B}}
\newcommand{\HyperSub}{\mathcal{C}}
\newcommand{\AffineHyperSub}{\mathcal{D}}
\newcommand{\ppointsX}{\mathcal{P}_X}
\newcommand{\llinesX}{\mathcal{L}_X}
\newcommand{\ppointsY}{\mathcal{P}_Y}
\newcommand{\llinesY}{\mathcal{L}_Y}

\renewcommand{\geq}{\geqslant}
\renewcommand{\leq}{\leqslant}
\renewcommand{\rho}{\varrho}
\renewcommand{\dim}[1]{\textnormal{dim}\left(#1\right)}

\title{Two disguises of the linear representation of a subgeometry}
\author{Lins Denaux \\ {\it Ghent University}}
\date{}

\begin{document}

\maketitle

\begin{abstract}
    Let $\pg(n,q)$ be the Desarguesian projective space of dimension $n$ over the finite field of order $q$.
    The \emph{linear representation} of a point set $\mathcal{K}$ in a hyperplane at infinity of $\pg(n,q)$ is the point-line geometry consisting of the affine points of $\pg(n,q)$, together with the union of the parallel classes of affine lines corresponding to the points of $\mathcal{K}$.
    This type of point-line geometry has been widely investigated in the literature.
    Curiously, if $\mathcal{K}$ is a subgeometry, two disguises of its linear representation occur in two separate works.
    In this short note, we give an explicit isomorphism between these two disguises by making use of field reduction.
\end{abstract}



\section{Introduction}\label{Sect_Introduction}

    The concept of a \emph{linear representation} was independently introduced for hyperovals by Ahrens and Szekeres \cite{AhrensSzekeres} and Hall \cite{Hall}, and extended to general point sets by De Clerck \cite{DeClerck}.
    
    \begin{df}\label{Def_LinearRep}
	    Let $s\in\NN$ and consider a point set $\mathcal{K}$ contained in a hyperplane $H_\infty$ of $\pg(s+1,q)$.
	    The \emph{linear representation} of $\mathcal{K}$ is the point-line geometry $T^*(\mathcal{K}):=(\mathscr{P}_\mathcal{K},\mathscr{L}_\mathcal{K})$ with natural incidence, where
	    \begin{itemize}
	        \item $\mathscr{P}_\mathcal{K}$ is the set of points of $\pg(s+1,q)\setminus H_\infty$, and
	        \item $\mathscr{L}_\mathcal{K}$ is the set of all point sets $\ell\setminus\{P\}$, where $\ell\nsubseteq H_\infty$ is a line of $\pg(s+1,q)$ through a point $P\in\mathcal{K}$.
	    \end{itemize}
	\end{df}
	
	Note that, essentially, the linear representation of $\mathcal{K}$ is a point-line geometry consisting of the affine points of $\pg(s+1,q)$, together with a union of parallel classes of its affine lines; the set $\mathcal{K}$ determines which parallel classes are considered.
	
	Linear representations are mainly investigated in case of $\mathcal{K}$ being a well-known object, such as a hyperoval (if $s=2$ and $q$ is even, see \cite{BicharaMazzoccaSomma,GrundhoferJoswigStroppel}), a Buekenhout-Metz unital (if $s=2$ and $q$ is square, see \cite{DeWinter}) or a subgeometry (see \cite{DeClerck,DeWinterRotteyVandeVoorde,Debroey}); results on general linear representations were proven as well, see e.g.\ \cite{BaderLunardon,CaraRotteyVandeVoorde,DeWinterRotteyVandeVoorde}.
	The introduction of \cite{DeWinterRotteyVandeVoorde} presents a nice overview of these known results.
    
    \bigskip
    This note focuses on the case of $\mathcal{K}$ being a \emph{subgeometry} (see Definition \ref{Def_Subgeometry}).
    As a matter of fact, if $\mathcal{K}\cong\pg(s,q)$, the point-line geometry $T^*(\mathcal{K})$ embedded in $\pg\big(s+1,q^t\big)$, $t\in\NNnot$, occurs somewhat \emph{disguised} in two separate works.
    On the one hand, the authors of \cite{DeWinterRotteyVandeVoorde} introduce the point-line geometry $X(s,t,q)$, its `points' being $(t-1)$-dimensional subspaces of $\pg(s+t,q)$ disjoint to a fixed $s$-space and its `lines' being $t$-dimensional subspaces intersecting that same $s$-space in a point (see Definition \ref{Def_PointLineX}).
    On the other hand, the author of \cite{Denaux} introduces the point-line geometry $Y(s,t,q)$, its `points' being the affine points of $\pg\big(t,q^{s+1}\big)$ and its `lines' being the affine parts of $\FF_q$-subgeometries of maximal dimension that share a fixed subgeometry of maximal dimension within the hyperplane at infinity (see Definition \ref{Def_PointLineY}).
    In both works, the authors prove that their newly introduced point-line geometry is in fact isomorphic to the linear representation $T^*(\mathcal{K})$ embedded in $\pg\big(s+1,q^t\big)$, with $\mathcal{K}\cong\pg(s,q)$; explicit isomorphisms are provided in both cases.
    
    This observation triggers the question whether an explicit, direct isomorphism between the point-line geometries $X(s,t,q)$ and $Y(s,t,q)$ can be obtained as well.
    The answer turns out to be positive, and hence the aim of this note is to describe and prove this explicit isomorphism.
    The isomorphism in question is constructed by making use of the technique called \emph{field reduction}.
	
	\section{Preliminaries}
	
	Throughout this note, we assume $q$ to be an arbitrary prime power.
	The Galois field of order $q$ will be denoted by $\FF_q$ and the Desarguesian projective space of (projective) dimension $n\in\NN$ over $\FF_q$ will be denoted by $\pg(n,q)$.
	
	\begin{df}\label{Def_Subgeometry}
		Let $s\in\NN$ and $t\in\{0,1,\dots,n\}$.
		A \emph{$t$-dimensional} \emph{$\FF_q$-subgeometry} $\Sub$ of $\pg\big(n,q^{s+1}\big)$ is a set of subspaces (points, lines, \ldots, $(t-1)$-dimensional subspaces) of $\pg\big(n,q^{s+1}\big)$, together with the incidence relation inherited from $\pg\big(n,q^{s+1}\big)$, such that $\Sub$ is isomorphic to $\pg(t,q)$.
		
		If $t=1$ or $t=2$, we will often call $\Sub$ an \emph{$\FF_q$-subline} or an \emph{$\FF_q$-subplane} of $\pg(n,q)$, respectively.
	\end{df}
	
	We now present the definitions of the two relevant point-line geometries $X(s,t,q)$ and $Y(s,t,q)$ (see Section \ref{Sect_Introduction}).
	
	\begin{df}[{\cite[Section $3$]{DeWinterRotteyVandeVoorde}}]\label{Def_PointLineX}
	    Let $s\in\NN$ and $t\in\NNnot$.
	    Consider an $s$-dimensional subspace $\pi$ of $\pg(s+t,q)$.
	    The point-line geometry $X(s,t,q)$ is the incidence structure $(\ppointsX,\llinesX)$ with natural incidence, where
	    \begin{itemize}
	        \item $\ppointsX$ is the set of all $(t-1)$-spaces of $\pg(s+t,q)$ disjoint to $\pi$, and
	        \item $\llinesX$ is the set of all $t$-spaces of $\pg(s+t,q)$ meeting $\pi$ exactly in one point.
	    \end{itemize}
	\end{df}
	
	\begin{df}[{\cite[Definition $6.1.1$]{Denaux}}]\label{Def_PointLineY}
	    Let $s\in\NN$ and $t\in\NNnot$.
	    Consider a hyperplane $\Sigma$ of $\pg\big(t,q^{s+1}\big)$ containing a $(t-1)$-dimensional $\FF_q$-subgeometry $\HyperSub$.
	    The point-line geometry $Y(s,t,q)$ is the incidence structure $(\ppointsY,\llinesY)$ with natural incidence, where
	    \begin{itemize}
	        \item $\ppointsY$ is the set of points of $\pg\big(t,q^{s+1}\big)\setminus\Sigma$, and
	        \item $\llinesY$ is the set of all point sets $\Sub\setminus\HyperSub$, where $\Sub$ is a $t$-dimensional $\FF_q$-subgeometry of $\pg\big(t,q^{s+1}\big)$ that contains $\HyperSub$.
	    \end{itemize}
	\end{df}
	
	As described in Section \ref{Sect_Introduction}, the literature provides us with the following two results.
	
	\begin{res}[{\cite[Theorem $4.1$]{DeWinterRotteyVandeVoorde}}]
	    Let $s\in\NN$ and $t\in\NNnot$.
	    Let $\AffineHyperSub_{s,t,q}$ be an $s$-dimensional $\FF_q$-subgeometry of $\pg\big(s,q^t\big)$.
	    Then the point-line geometries $X(s,t,q)$ and $T^*(\AffineHyperSub_{s,t,q})$ are isomorphic.
	\end{res}
	
	\begin{res}[{\cite[Theorem $6.2.4$]{Denaux}}]
	    Let $s\in\NN$ and $t\in\NNnot$.
	    Let $\AffineHyperSub_{s,t,q}$ be an $s$-dimensional $\FF_q$-subgeometry of $\pg\big(s,q^t\big)$.
	    Then the point-line geometries $Y(s,t,q)$ and $T^*(\AffineHyperSub_{s,t,q})$ are isomorphic.
	\end{res}
	
	Hence, we obtain the following.
	
	\begin{crl}
	    Let $s\in\NN$ and $t\in\NNnot$.
	    Then the point-line geometries $X(s,t,q)$ and $Y(s,t,q)$ are isomorphic.
	\end{crl}
	
	This work focuses on an explicit isomorphism between $X(s,t,q)$ and $Y(s,t,q)$ by making use of a technique called \emph{field reduction}.
	
	\subsection*{Field reduction}
	
	Let $s,t\in\NN$.
	The idea behind field reduction is interpreting a projective geometry $\pg\big(t,q^{s+1}\big)$ as its underlying vector space $V\big(t+1,q^{s+1}\big)$, which is known to be isomorphic to $V(st+s+t+1,q)$, which in turn naturally translates to $\pg(st+s+t,q)$.
	In this way, one obtains a correspondence between subspaces of $\pg\big(t,q^{s+1}\big)$ and subspaces of $\pg(st+s+t,q)$ by `reducing' the underlying field.
	A great survey on this topic can be found in \cite{LavrauwVandeVoorde}.
	The authors of this work formally introduce the \emph{field reduction map}
	\begin{equation}\label{Eq_FieldReduction}
	    \FieldRed_{t+1,s+1,q}:\pg\big(t,q^{s+1}\big)\rightarrow\pg(st+s+t,q)\textnormal{,}
	\end{equation}
	which maps subspaces onto subspaces by viewing these as embedded projective geometries and applying field reduction.
	In favor of simplicity, we won't go into detail about the mechanics behind this map, but will profit from its many properties.
	We list the ones we need for this note.
	
	\begin{lm}[{\cite[Lemma $2.2$]{LavrauwVandeVoorde}}]\label{Lm_FieldRedProps}
	    Let $\mathcal{P}$ denote the set of points of $\pg\big(t,q^{s+1}\big)$ and consider $\FieldRed_{t+1,s+1,q}$ as defined in \eqref{Eq_FieldReduction}.
	    \begin{enumerate}
	        \item The field reduction map $\FieldRed_{t+1,s+1,q}$ is injective and preserves inclusion and disjointness of subspaces.
	        \item If $\kappa$ is a $k$-dimensional subspace of $\pg\big(t,q^{s+1}\big)$, $k\in\{-1,0,\dots,t\}$, then $\FieldRed_{t+1,s+1,q}(\kappa)$ has dimension $sk+s+k$.
	    \end{enumerate}
	\end{lm}
	
	Note that we have altered, omitted and generalised some properties stated in the original lemma (\cite[Lemma $2.2$]{LavrauwVandeVoorde}), but all the properties stated above follow immediately from the definition of $\FieldRed_{t+1,s+1,q}$.
	
	\bigskip
	The authors of \cite{LavrauwVandeVoorde} also prove that the image set of the set of points of a $k$-dimensional $\FF_q$-subgeometry of $\pg\big(t,q^{s+1}\big)$ under $\FieldRed_{t+1,s+1,q}$ is projectively equivalent to the system of $s$-spaces of a so-called \emph{Segre variety} $\mathcal{S}_{k,s}$ of the Segre variety $\mathcal{S}_{t,s}$.
	
	\begin{df}\label{Def_Regulus}
	    Let $s\in\NN$.
	    A \emph{regulus} $\Regulus$ of $\pg(2s+1,q)$ is a set of $q+1$ pairwise disjoint $s$-spaces, with the property that any line meeting three elements of $\Regulus$, intersects all elements of $\Regulus$.
	    Such a line will be called a \emph{transversal line} of $\Regulus$.
	\end{df}
	
	As a regulus is actually one of the systems of a Segre variety $\mathcal{S}_{1,s}$, we obtain the following.
	
	\begin{lm}[{\cite[Theorem $2.6$]{LavrauwVandeVoorde}}]\label{Lm_SublineIsRegulus}
	    The image set of the set of points on an $\FF_q$-subline of $\pg\big(1,q^{s+1}\big)$ under $\FieldRed_{2,s+1,q}$ is a regulus of $\pg(2s+1,q)$.
	\end{lm}
	
	\section{Reguli and transversal lines}
	
	This section aims to prove Lemma \ref{Lm_ExistsMaximallyIntersectingInGeneralisedRegulus}, which will turn out to be crucial to prove the main theorem (Theorem \ref{Thm_Main}).
	The proof of Lemma \ref{Lm_ExistsMaximallyIntersectingInGeneralisedRegulus} will be done with the help of some elementary lemmas of geometric nature.
	
	\begin{lm}\label{Lm_ThreeDisjointIntersectLine}
	    Let $s\in\NN$.
	    Consider three pairwise disjoint $s$-spaces $\sigma_1$, $\sigma_2$ and $\sigma_3$ of $\pg(2s+1,q)$ and suppose that $S$ is a point contained in $\sigma_1$.
	    Then there exists a unique line through $S$ intersecting each $\sigma_i$ exactly in a point.
	\end{lm}
	\begin{proof}
	    The $(s+1)$-dimensional space $\vspan{S,\sigma_2}$ necessarily intersects $\sigma_3$ exactly in a point $T$.
	    In fact, any line through $S$ meeting both $\sigma_2$ and $\sigma_3$ naturally lies in $\vspan{S,\sigma_2}$ and hence has to contain $T$.
	    As the line $\vspan{S,T}$ lies in the $(s+1)$-space $\vspan{S,\sigma_2}$, it has to intersect $\sigma_2$.
	    This is the unique line we are looking for.
	\end{proof}
	
	Consider a regulus $\Regulus$ of $\pg(2s+1,q)$ and suppose that $S$ is a point contained in an element of $\Regulus$.
	Then Lemma \ref{Lm_ThreeDisjointIntersectLine} directly implies that there exists a unique transversal line of $\Regulus$ through $S$ (see Definition \ref{Def_Regulus}).
	Therefore, we will often speak of \emph{the} transversal line of $\Regulus$ through $S$.
	
	
	\begin{lm}\label{Lm_ExistsMaximallyIntersectingInRegulus}
	    Let $s\in\NN$, let $k\in\{1,2,\dots,s+1\}$ and consider a regulus $\Regulus$ of $\pg(2s+1,q)$.
	    Let $R_1,R_2,\dots,R_k$ be $k$ independent points of an element of $\Regulus$, with $t_1,t_2,\dots,t_k$ their corresponding transversal lines.
	    Then $\TransversalSpace:=\vspan{t_1,t_2,\dots,t_k}$ is a $(2k-1)$-space intersecting each element of $\Regulus$ exactly in a $(k-1)$-space.
	\end{lm}
	\begin{proof}
	    Suppose w.l.o.g.\ that $R_1,R_2,\dots,R_k\in\sigma_1\in\Regulus$.
	    
	    We will prove the statement by induction on $k$.
	    If $k=1$, then the proof is done as $\TransversalSpace$ is a transversal line.
	    Hence, we can assume that $k>1$ and that $\TransversalSpace':=\vspan{t_1,t_2,\dots,t_{k-1}}$ is a $(2k-3)$-space intersecting each element $\sigma_i\in\Regulus$ exactly in a $(k-2)$-space $\sigma_i'$, $i\in\{1,2,\dots,q+1\}$.
	    
	    Note that the transversal line $t_k$ intersects at most one element of $\{\sigma_1',\sigma_2',\dots,\sigma_{q+1}'\}$.
	    Indeed, if $t_k$ would meet two distinct elements $\sigma_{i_1}'$ and $\sigma_{i_2}'$, then $t_k$ would lie in $\TransversalSpace'$ and, as a result, $R_k$ would lie in the $(k-2)$-space $\sigma_1'=\vspan{R_1,R_2,\dots,R_{k-1}}$, contradicting the fact that $R_1,R_2,\dots,R_k$ are independent points.
	    
	    Hence, suppose, to the contrary and w.l.o.g., that the transversal line $t_k$ intersects $\sigma_2'$ and is disjoint to $\sigma_1'$ and $\sigma_3'$.
	    As $t_k$ intersects $\sigma_2'$, $\dim{\vspan{\sigma_1',\sigma_2',t_k}}\leq(k-2)+(k-2)+1+1=2k-2$; conversely, as $\vspan{\sigma_1',\sigma_2'}$ intersects $\sigma_1$ in $\sigma_1'$ and as $t_k$ intersects $\sigma_1\setminus\sigma_1'$, we obtain $\dim{\vspan{\sigma_1',\sigma_2',t_k}}\geq\dim{\vspan{\sigma_1',\sigma_2'}}+1=2k-2$ and hence $\dim{\vspan{\sigma_1',\sigma_2',t_k}}=2k-2$.
	    As $\vspan{\sigma_1',\sigma_2'}=\mathcal{T}'$, we conclude that $\dim{\vspan{\mathcal{T}',t_k}}=2k-2$.
	    However, both the $(k-1)$-space $\vspan{\sigma_1',R_k}$ and the $(k-1)$-space $\vspan{\sigma_3',t_k\cap\sigma_3}$ are contained in the $(2k-2)$-space $\vspan{\mathcal{T}',t_k}$ while also being disjoint to each other, a contradiction.
	    
	    In conclusion, the transversal line $t_k$ does not intersect any element of the set $\{\sigma_1',\sigma_2',\dots,\sigma_{q+1}'\}$.
	    Hence, $\TransversalSpace=\vspan{\TransversalSpace',t_k}$ intersects each space $\sigma_i$ at least in a $(k-1)$-space, but also at most, as else we can find two disjoint spaces in $\TransversalSpace$ that span a $(k-1)+k+1=(2k)$-space, in contradiction with $\dim{\TransversalSpace}\leq\dim{\TransversalSpace'}+1+1=2k-1$.
	    The fact that $\TransversalSpace$ contains two disjoint $(k-1)$-spaces implies that the latter inequality is in fact an equality and finishes the proof.
	\end{proof}
	
	For the lemma below, as well as for the proof of Theorem \ref{Thm_Main}, we define $\FieldRed_{t+1,s+1,q}(\mathcal{P})$ to be the set of images of all points of $\mathcal{P}$ under $\FieldRed_{t+1,s+1,q}$, where $\mathcal{P}$ is an arbitrary point set of $\pg\big(t,q^{s+1}\big)$,
	
	\begin{lm}\label{Lm_ExistsMaximallyIntersectingInGeneralisedRegulus}
	    Let $s\in\NN$ and $t\in\NNnot$.
	    Consider a $(t-1)$-dimensional $\FF_q$-subgeometry $\HyperSub$ of $\pg\big(t-1,q^{s+1}\big)$ and let $\mathcal{P}_\HyperSub$ be its point set.
	    Then there exists an $(st-1)$-dimensional space of $\pg(st+t-1,q)$ intersecting each element of the set $\FieldRed_{t,s+1,q}(\mathcal{P}_\HyperSub)$ exactly in an $(s-1)$-space.
	\end{lm}
	\begin{proof}
	    In this proof, we will extend the notation $\mathcal{P}_\mathcal{A}$ as being the point set of any $\FF_q$-subgeometry $\mathcal{A}$.
	    
	    The proof will be done by induction on $t$.
	    If $t=1$, the statement is trivially true.
	    If $t=2$, then $\HyperSub$ is an $\FF_q$-subline and hence, by Lemma \ref{Lm_SublineIsRegulus}, $\FieldRed_{2,s+1,q}(\mathcal{P}_\HyperSub)$ is a regulus of $\pg(2s+1,q)$.
	    Thus, the statement follows from Lemma \ref{Lm_ExistsMaximallyIntersectingInRegulus} by choosing $s$ independent points in an element of $\FieldRed_{2,s+1,q}(\mathcal{P}_\HyperSub)$.
	    Therefore, we may assume that $t>2$ and consider a $(t-2)$-dimensional $\FF_q$-subgeometry $\HyperSub'\subseteq\HyperSub$ (spanning a $(t-2)$-space $\Sigma$) for which there exists an $(st-s-1)$-dimensional subspace $\TransversalSpace_{\HyperSub'}$ of $\FieldRed_{t,s+1,q}(\Sigma)$ intersecting each element of the set $\FieldRed_{t,s+1,q}(\mathcal{P}_{\HyperSub'})$ exactly in an $(s-1)$-space.
	    
	    Now
	    \begin{itemize}
	        \item define $\Pi_\Sigma:=\FieldRed_{t,s+1,q}(\Sigma)$,
	        \item let $Q\in\HyperSub'$ and define $\sigma_Q:=\FieldRed_{t,s+1,q}(Q)$ and $\sigma_Q':=\TransversalSpace_{\HyperSub'}\cap\sigma_Q$,
	        \item consider an $\FF_q$-subline $\mathfrak{L}$ of $\HyperSub$ through $Q$, not contained in $\HyperSub'$, and let $\Pi_\mathfrak{L}$ be the image of its span under $\FieldRed_{t,s+1,q}$,
	        \item define $\Regulus_\mathfrak{L}:=\FieldRed_{t,s+1,q}(\mathcal{P}_\mathfrak{L})$.
	    \end{itemize}
	    By  Lemma \ref{Lm_SublineIsRegulus}, $\Regulus_\mathfrak{L}$ is a regulus contained in the $(2s+1)$-dimensional space $\Pi_\mathfrak{L}$.
	    Hence, by Lemma \ref{Lm_ExistsMaximallyIntersectingInRegulus}, we can consider a $(2s-1)$-dimensional subspace $\TransversalSpace_\mathfrak{L}$ in $\Pi_\mathfrak{L}$ through $\sigma_Q'$ intersecting each element of $\Regulus_\mathfrak{L}$ exactly in an $(s-1)$-space.
	    Moreover, by Lemma \ref{Lm_FieldRedProps}, the image $\FieldRed_{t,s+1,q}(R)$ of any point $R\in\mathfrak{L}\setminus\HyperSub'$ is disjoint to $\Pi_\Sigma$, which implies that the $(2s+1)$-space $\Pi_\mathfrak{L}$ intersects the $(st-s+t-2)$-space $\Pi_\Sigma$ precisely in the $s$-space $\sigma_Q$.
	    As a consequence, $\TransversalSpace_\mathfrak{L}$ intersects $\TransversalSpace_{\HyperSub'}$ exactly in the $(s-1)$-space $\sigma_Q'$, hence
	    \[
	        \TransversalSpace:=\vspan{\TransversalSpace_\mathfrak{L},\TransversalSpace_{\HyperSub'}}
	    \]
	    has dimension $(2s-1)+(st-s-1)-(s-1)=st-1$.
	    We will prove that $\TransversalSpace$ is the $(st-1)$-space of $\pg(st+t-1,q)$ we are looking for.
	    
	    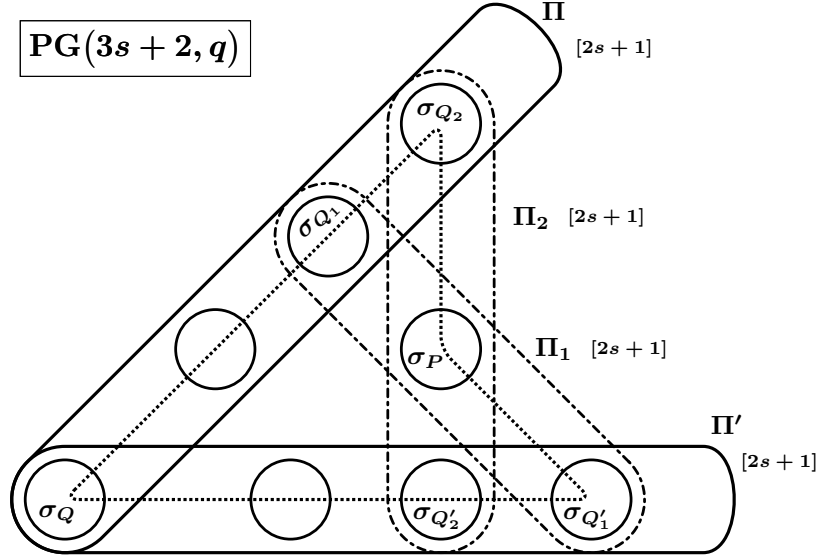
\begin{figure}
	        \begin{center}\begin{tikzpicture}
	            
	            \node[draw,fill=none,anchor=west] at (-3.6,3) {\large\textbf{PG}$\boldsymbol{\big(3s+2,q\big)}$};
    			\node[draw,fill=none,anchor=west] at (-3.6,3) {\large\textbf{PG}$\boldsymbol{\big(3s+2,q\big)}$}; 
    			
    			\draw[line width=1.2pt, line join=round, line cap=round] (-3.5,-2.5) -- (2.5,3.5) to [out=45, in=45] (3.5,2.5) -- (-2.5,-3.5) to [out=225, in=315] (-3.5,-3.5) to [out=135, in=225] cycle;
    			\node[draw=none,fill=none] at (3.5,3.5) {$\boldsymbol{\Pi}$};
    			\node[draw=none,fill=none] at (4.3,3) {\scriptsize$\boldsymbol{[2s+1]}$};
    			
    			\draw[line width=1.2pt, line join=round, line cap=round] (-3,-2.293) -- (5.485,-2.293) to [out=0, in=0] (5.485,-3.707) -- (-3,-3.707) to [out=180, in=270] (-3.707,-3) to [out=90, in=180] cycle;
    			\node[draw=none,fill=none] at (5.8,-2) {$\boldsymbol{\Pi'}$};
    			\node[draw=none,fill=none] at (6.5,-2.5) {\scriptsize$\boldsymbol{[2s+1]}$};
    			
    			\draw[line width=1.05pt] (-3,-3) circle (15pt);
    			\node[draw=none,fill=none] at (-3.1,-3.2) {$\boldsymbol{\sigma_Q}$};
    			
    			\draw[line width=1.05pt] (2,-1) circle (15pt);
    			\node[draw=none,fill=none] at (1.8,-1.15) {$\boldsymbol{\sigma_P}$};

    			\draw[line width=1.05pt] (0.5,0.5) circle (15pt);
    			\node[draw=none,fill=none, rotate around={45:(0,0.55)}] at (0,0.55) {$\boldsymbol{\sigma_{Q_1}}$};
    			\draw[line width=1.05pt] (2,2) circle (15pt);
    			\node[draw=none,fill=none] at (2,2.15) {$\boldsymbol{\sigma_{Q_2}}$};

    			\draw[line width=1.05pt] (2,-3) circle (15pt);
    			\node[draw=none,fill=none] at (1.95,-3.25) {$\boldsymbol{\sigma_{Q_2'}}$};
    			\draw[line width=1.05pt] (4,-3) circle (15pt);
    			\node[draw=none,fill=none] at (3.95,-3.25) {$\boldsymbol{\sigma_{Q_1'}}$};
    			
    			\draw[line width=1.05pt] (-1,-1) circle (15pt);
    			\draw[line width=1.05pt] (0,-3) circle (15pt);
    			
    			\draw[line width=1.05pt, dashdotted, line join=round, line cap=round] (4.5,-2.5) -- (1,1) to [out=135, in=45] (0,1) to [out=225, in=135] (0,0) -- (3.5,-3.5) to [out=315, in=225] (4.5,-3.5) to [out=45, in=315] cycle;
    			\node[draw=none,fill=none] at (3.5,-1) {$\boldsymbol{\Pi_1}$};
    			\node[draw=none,fill=none] at (4.5,-1) {\scriptsize$\boldsymbol{[2s+1]}$};
    			
    			\draw[line width=1.05pt, dashdotted, line join=round, line cap=round] (2.707,-3) -- (2.707,2) to [out=90, in=0] (2,2.707) to [out=180, in=90] (1.293,2) -- (1.293,-3) to [out=270, in=180] (2,-3.707) to [out=0, in=270] cycle;
    			\node[draw=none,fill=none] at (3.2,0.75) {$\boldsymbol{\Pi_2}$};
    			\node[draw=none,fill=none] at (4.2,0.75) {\scriptsize$\boldsymbol{[2s+1]}$};
    			
    			\draw[line width=1.2pt, rounded corners=5pt, densely dotted] (-3,-3) -- (2,2) -- (2,-1) -- (4,-3) -- cycle;
    			
	        \end{tikzpicture}\end{center}
	        \caption{A visualisation of the proof of Lemma \ref{Lm_ExistsMaximallyIntersectingInGeneralisedRegulus}, or how $\vspan{\TransversalSpace_\mathfrak{L},\TransversalSpace_{\mathfrak{L}'}}$ should intersect $\sigma_P$ in an $(s-1)$-space, $P\in\vspan{\mathfrak{L},\mathfrak{L}'}\setminus\left(\mathfrak{L}\cup\mathfrak{L}'\right)$. All circles are $s$-spaces.}
	        \label{Fig_ExistenceOfChi}
	    \end{figure}
	    
	    \bigskip
	    Choose an arbitrary $P\in\HyperSub$ and define $\sigma_P:=\FieldRed_{t,s+1,q}(P)$.
	    The only thing left to prove is that $\TransversalSpace$ intersects $\sigma_P$ exactly in an $(s-1)$-space.
	    Note that $\TransversalSpace$ intersects $\Pi_\mathfrak{L}$ and $\Pi_\Sigma$ at least in $\TransversalSpace_\mathfrak{L}$ and $\TransversalSpace_{\HyperSub'}$, respectively, but also \textit{at most}, as else we can use Grassmann's identity to prove that $\dim{\TransversalSpace}>st-1$, a contradiction.
	    If $P\in\mathfrak{L}\cup\HyperSub'$, then $\sigma_P$ is contained in either $\Pi_\mathfrak{L}$ or $\Pi_\Sigma$, hence $\TransversalSpace$ will intersect $\sigma_P$ exactly in an $(s-1)$-space.
	    
	    Now suppose $P\notin\mathfrak{L}\cup\HyperSub'$.
	    Let $\mathfrak{P}$ be the $\FF_q$-subplane of $\HyperSub$ spanned by $\mathfrak{L}$ and $P$; this subplane intersects $\HyperSub'$ in an $\FF_q$-subline $\mathfrak{L}'$ through $Q$; let $\Pi_{\mathfrak{L}'}$ be the image under $\FieldRed_{t,s+1,q}$ of the span of $\mathfrak{L}'$.
	    Note that by the above arguments, $\TransversalSpace$ will intersect $\Pi_{\mathfrak{L}'}$ in an $(2s-1)$-space $\TransversalSpace_{\mathfrak{L}'}$, in a similar way as it intersects $\Pi_\mathfrak{L}$.
	    Hence, we can focus on the $(3s+2)$-space $\vspan{\Pi_\mathfrak{L},\Pi_{\mathfrak{L}'}}\supseteq\sigma_P$ to continue this proof (see Figure \ref{Fig_ExistenceOfChi}).
	    
	    Let $\mathfrak{L}_1$ and $\mathfrak{L}_2$ be two distinct $\FF_q$-sublines in $\mathfrak{P}$ through $P$, not containing $Q$.
	    Define
	    \[
	        Q_1:=\mathfrak{L}\cap\mathfrak{L}_1\textnormal{,}\quad Q_2:=\mathfrak{L}\cap\mathfrak{L}_2\textnormal{,}\quad Q_1':=\mathfrak{L}'\cap\mathfrak{L}_1\textnormal{,}\quad Q_2':=\mathfrak{L}'\cap\mathfrak{L}_2\textnormal{.}
	    \]
	    Correspondingly, define
	    \[
	        \sigma_{Q_1}:=\FieldRed_{t,s+1,q}(Q_1)\textnormal{,}\quad \sigma_{Q_2}:=\FieldRed_{t,s+1,q}(Q_2)\textnormal{,}\quad \sigma_{Q_1'}:=\FieldRed_{t,s+1,q}(Q_1')\textnormal{,}\quad \sigma_{Q_2'}:=\FieldRed_{t,s+1,q}(Q_2')\textnormal{,}
	    \]
	    and
	    \[
	        \sigma_{Q_1}':=\TransversalSpace\cap\sigma_{Q_1}\textnormal{,}\quad \sigma_{Q_2}':=\TransversalSpace\cap\sigma_{Q_2}\textnormal{,}\quad \sigma_{Q_1'}':=\TransversalSpace\cap\sigma_{Q_1'}\textnormal{,}\quad \sigma_{Q_2'}':=\TransversalSpace\cap\sigma_{Q_2'}\textnormal{.}
	    \]
	    As $Q_1,Q_2,Q_1',Q_2'\in\mathfrak{L}\cup\mathfrak{L}'\subseteq\mathfrak{L}\cup\HyperSub'$, we know that the above spaces are $(s-1)$-spaces.
	    Finally, define
	    \[
	        \Pi:=\vspan{\sigma_{Q_1},\sigma_{Q_2}}=\Pi_\mathfrak{L}\textnormal{,}\quad
	        \Pi':=\vspan{\sigma_{Q_1'},\sigma_{Q_2'}}=\Pi_{\mathfrak{L}'}\textnormal{,}\quad \Pi_1:=\vspan{\sigma_{Q_1},\sigma_{Q_1'}}\textnormal{,}\quad \Pi_2:=\vspan{\sigma_{Q_2},\sigma_{Q_2'}}\textnormal{,}
	    \]
	    these all being $(2s+1)$-spaces of $\pg(st+t-1,q)$.
	    Note that $\sigma_Q$ is contained in both $\Pi$ and $\Pi'$.
	    
	    Observe that, as $\Pi\cap\Pi'=\sigma_Q$ and as both $\TransversalSpace_\mathfrak{L}\subseteq\Pi$ and $\TransversalSpace_{\mathfrak{L}'}\subseteq\Pi'$ intersect $\sigma_Q$ in the $(s-1)$-space $\sigma_Q'$, the intersection $\TransversalSpace_\mathfrak{L}\cap\TransversalSpace_{\mathfrak{L}'}$ has dimension $s-1$, hence $\dim{\vspan{\TransversalSpace_\mathfrak{L},\TransversalSpace_{\mathfrak{L}'}}}=3s-1$.
	    
	    We can prove that $\vspan{\TransversalSpace_\mathfrak{L},\TransversalSpace_{\mathfrak{L}'}}\cap\Pi_1=\vspan{\sigma_{Q_1}',\sigma_{Q_1'}'}$.
	    Indeed, we know that $\Pi_1=\vspan{\sigma_{Q_1},\sigma_{Q_1'}}$; as $\TransversalSpace_\mathfrak{L}$ intersects $\sigma_{Q_1}$ in $\sigma_{Q_1}'$ and as $\TransversalSpace_{\mathfrak{L}'}$ intersects $\sigma_{Q_1'}$ in $\sigma_{Q_1'}'$, the space $\vspan{\TransversalSpace_\mathfrak{L},\TransversalSpace_{\mathfrak{L}'}}\cap\Pi_1$ contains $\vspan{\sigma_{Q_1}',\sigma_{Q_1'}'}$.
	    If $\vspan{\TransversalSpace_\mathfrak{L},\TransversalSpace_{\mathfrak{L}'}}$ would contain a subspace of $\Pi_1$ of dimension larger than $\dim{\vspan{\sigma_{Q_1}',\sigma_{Q_1'}'}}$, then $\vspan{\TransversalSpace_\mathfrak{L},\TransversalSpace_{\mathfrak{L}'}}$ would contain both the $(s-1)$-space $\sigma_Q'\subseteq\sigma_Q$ and a $(2s)$-subspace of $\Pi_1$, which are disjoint to each other as $\sigma_Q$ and $\Pi_1$ are disjoint by Lemma \ref{Lm_FieldRedProps}.
	    This would imply that $\dim{\vspan{\TransversalSpace_\mathfrak{L},\TransversalSpace_{\mathfrak{L}'}}}\geq(s-1)+2s+1=3s$, a contradiction.
	    
	    As a consequence, $\TransversalSpace$ cannot contain $\sigma_P$, as else the $(2s-1)$-space $\vspan{\TransversalSpace_\mathfrak{L},\TransversalSpace_{\mathfrak{L}'}}\cap\Pi_1$ contains both the $(s-1)$-space $\sigma_{Q_1}'$ and the $s$-space $\sigma_P$, which are disjoint to each other, a contradiction.
	    Hence, $\TransversalSpace$ intersects $\sigma_P$ at most in an $(s-1)$-space.
	    It remains to prove that $\TransversalSpace$ intersects $\sigma_P$ \emph{at least} in an $(s-1)$-space
	    
	    As $\Pi\cap\Pi'=\sigma_Q$, we have that
	    \[
	        \dim{\vspan{\Pi_1,\Pi_2}}=\dim{\vspan{\sigma_{Q_1},\sigma_{Q_1'},\sigma_{Q_2},\sigma_{Q_2'}}}=\dim{\vspan{\Pi,\Pi'}}=3s+2\textnormal{.}
        \]
        Hence, $\dim{\Pi_1\cap\Pi_2}=(2s+1)+(2s+1)-(3s+2)=s$.
        As, by Lemma \ref{Lm_FieldRedProps}, the $s$-space $\sigma_P$ is contained in both $\Pi_1$ and $\Pi_2$, this means that $\Pi_1\cap\Pi_2=\sigma_P$ (see Figure \ref{Fig_ExistenceOfChi}).
        
        Recall that $\TransversalSpace_\mathfrak{L}=\vspan{\sigma_{Q_1}',\sigma_{Q_2}'}$, $\TransversalSpace_\mathfrak{L'}=\vspan{\sigma_{Q_1'}',\sigma_{Q_2'}'}$ and that the span $\vspan{\TransversalSpace_\mathfrak{L},\TransversalSpace_{\mathfrak{L}'}}$ is a $(3s-1)$-dimensional space.
        Hence, we can make a similar reasoning as above and obtain that
        \begin{align*}
	        \dim{\vspan{\sigma_{Q_1}',\sigma_{Q_1'}'}\cap\vspan{\sigma_{Q_2}',\sigma_{Q_2'}'}}&=2(2s-1)-\dim{\vspan{\sigma_{Q_1}',\sigma_{Q_1'}',\sigma_{Q_2}',\sigma_{Q_2'}'}}\\
	        &=2(2s-1)-\dim{\vspan{\TransversalSpace_\mathfrak{L},\TransversalSpace_{\mathfrak{L}'}}}\\
	        &=s-1\textnormal{.}
        \end{align*}
        As $\vspan{\sigma_{Q_1}',\sigma_{Q_1'}'}\subseteq\Pi_1$ and $\vspan{\sigma_{Q_2}',\sigma_{Q_2'}'}\subseteq\Pi_2$, the $(s-1)$-space $\vspan{\sigma_{Q_1}',\sigma_{Q_1'}'}\cap\vspan{\sigma_{Q_2}',\sigma_{Q_2'}'}$ lies in $\Pi_1\cap\Pi_2=\sigma_P$.
        Hence, $\TransversalSpace$ intersects $\sigma_P$ at least in this $(s-1)$-space, and the proof is done.
	\end{proof}
	
	\section{The isomorphism between $\boldsymbol{X(s,t,q)}$ and $\boldsymbol{Y(s,t,q)}$}
	
	We now have all the tools we need to construct an isomorphism between $X(s,t,q)$ and $Y(s,t,q)$.
	The point-line geometry $Y(s,t,q)$ is embedded in $\pg\big(t,q^{s+1}\big)$; in this section, we will use the same notation as introduced in Definition \ref{Def_PointLineY}.
    Consider the field reduction map \eqref{Eq_FieldReduction}:
    \[
        \FieldRed_{t+1,s+1,q}:\pg\big(t,q^{s+1}\big)\rightarrow\pg(st+s+t,q)\textnormal{.}
    \]
    To simplify notation, we define $\FieldRed:=\FieldRed_{t+1,s+1,q}$ throughout this section.
	We refer to Figure \ref{Fig_PointLineIsomorphism} for a visualisation of the isomorphism $\varphi$ we are about to define.
	
	\begin{df}\label{Def_IsomorphismPhi}
	    Let $s\in\NN$, $t\in\NNnot$ and consider the point-line geometry $Y(s,t,q)$ together with all corresponding notation (see Definition \ref{Def_PointLineY}).
	    By Lemma \ref{Lm_ExistsMaximallyIntersectingInGeneralisedRegulus} \big(and temporarily restricting the field reduction map to $\Sigma\cong\pg\big(t-1,q^{s+1}\big)$\big), we can consider an $(st-1)$-dimensional subspace $\chi$ of the $(st+t-1)$-space $\FieldRed(\Sigma)$ that intersects $\FieldRed(Q)$ exactly in an $(s-1)$-space, for every $Q\in\HyperSub$ (see Figure \ref{Fig_PointLineIsomorphism}).
	    
	    Now consider a duality $d$ of $\pg(st+s+t,q)$ and define the map
	    \[
	        \varphi:\pg\big(t,q^{s+1}\big)\rightarrow\pg(s+t,q):\tau\mapsto\left(\FieldRed(\tau)^d\cap\chi^d\right)\textnormal{,}
	    \]
	    where we identify $\chi^d\cong\pg(s+t,q)$ and where $\tau$ is a subspace of $\pg\big(t,q^{s+1}\big)$.
	\end{df}
	
	\begin{thm}\label{Thm_Main}
	    Let $s\in\NN$ and $t\in\NNnot$.
	    Then $\varphi$ induces an isomorphism between $Y(s,t,q)$ and $X(s,t,q)$.
	\end{thm}
	\begin{proof}
	    By the properties of a duality, $\chi\subseteq\FieldRed(\Sigma)$ implies $\FieldRed(\Sigma)^d\subseteq\chi^d$, hence we have that $\varphi(\Sigma)=\FieldRed(\Sigma)^d$.
	    In line with Definition \ref{Def_PointLineX}, we now
	    \begin{itemize}
	        \item identify $\chi^d$ with $\pg(s+t,q)$, and
	        \item define $\pi:=\varphi(\Sigma)=\FieldRed(\Sigma)^d$.
	    \end{itemize}
	    
	    \smallskip
	    \underline{Claim $1$}: $\varphi$ is a bijection between $\ppointsY$ and $\ppointsX$.
	    
	    \bigskip
	    Let $P$ be a point of $\ppointsY$.
	    As $P\notin\Sigma$, by Lemma \ref{Lm_FieldRedProps}, $\FieldRed(P)$ is an $s$-space disjoint to the $(st+t-1)$-space $\FieldRed(\Sigma)$.
	    Hence, $\FieldRed(P)^d$ is an $(st+t-1)$-space of $\pg(st+s+t,q)$ disjoint to the $s$-space $\FieldRed(\Sigma)^d=\varphi(\Sigma)=\pi$.
	    Therefore, $\chi^d$ intersects $\FieldRed(P)^d$ in a space of dimension at least $t-1$, as both spaces are contained in $\pg(st+s+t,q)$.
	    Conversely, $\chi^d$ intersects $\FieldRed(P)^d$ in a space of dimension at most $t-1$ as $\pi$ and $\FieldRed(P)^d\cap\chi^d$ are disjoint spaces that are both contained in the $(s+t)$-space $\chi^d$.
	    
	    As a result, $\varphi$ maps elements of $\ppointsY$ onto elements of $\ppointsX$.
	    By Lemma \ref{Lm_FieldRedProps}, $\varphi$ is injective.
	    To prove bijectivity, one can prove that $|\ppointsY|=|\ppointsX|$. The fact that
	    $|\ppointsY|=\big|\ag\big(t,q^{s+1}\big)\big|=q^{st+t}$ follows directly.
	    To prove that $|\ppointsX|=q^{st+t}$ as well, one has to count the number of $(t-1)$-spaces in $\pg(s+t,q)$ disjoint to a fixed $s$-space, see e.g.\ \cite[Sect. $170$]{Segre}.
	    
	    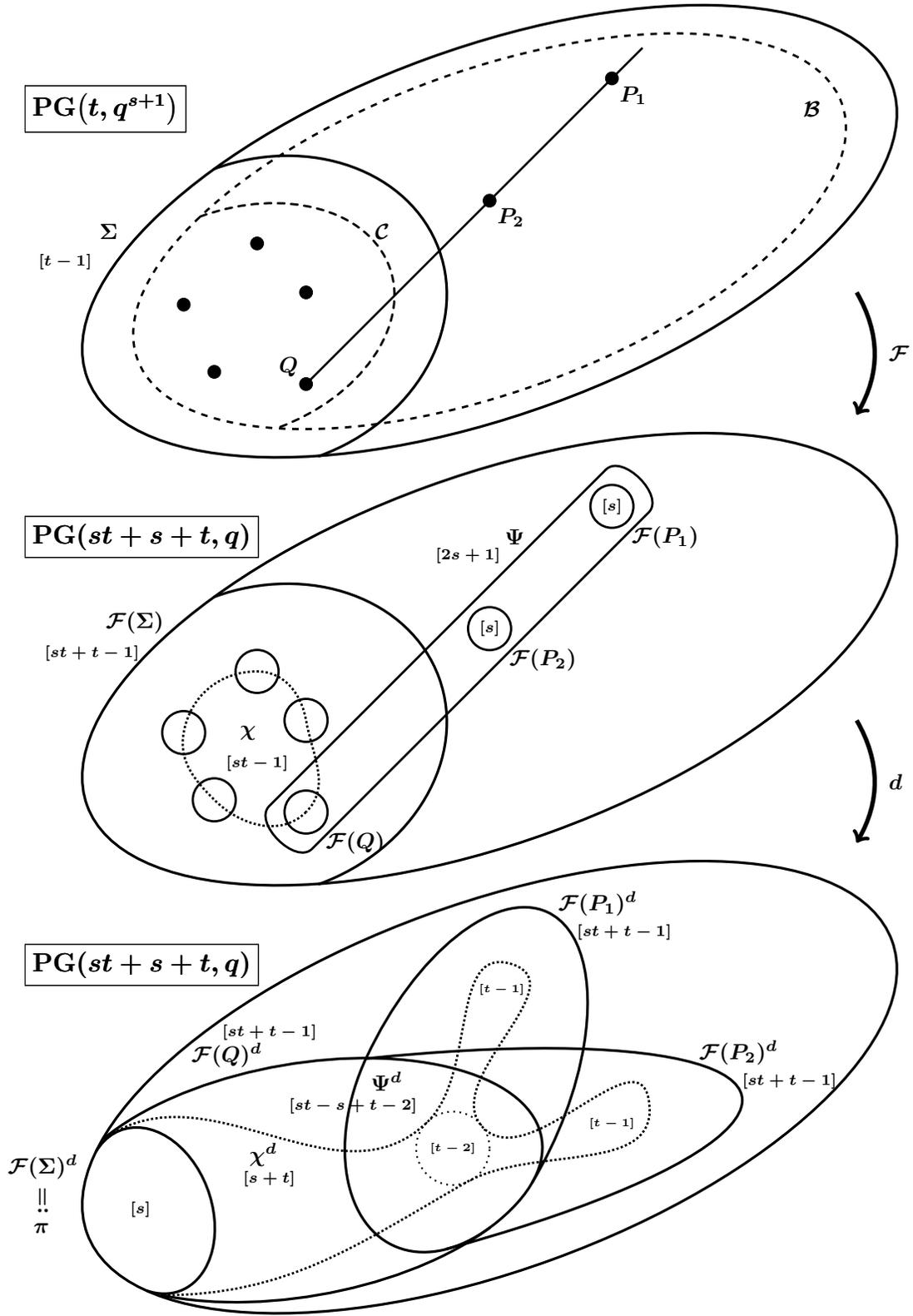
\begin{figure}
		\begin{center}\begin{tikzpicture}[scale=0.97722625]
		    
		    
			\draw[line width=1.2pt, rotate around={290:(0,0)}] (0,0) ellipse (3cm and 7cm);
			\node[draw,fill=none,anchor=west] at (-7.6,2) {\large\textbf{PG}$\boldsymbol{\big(t,q^{s+1}\big)}$};
			\node[draw,fill=none,anchor=west] at (-7.6,2) {\large\textbf{PG}$\boldsymbol{\big(t,q^{s+1}\big)}$}; 
			
			\draw[line width=1.2pt, rotate around={290:(-4.5,1.02)}, line join=round, line cap=round] (-4.5,1.02) arc (180:0:2.5cm and 3cm);
			\node[draw=none,fill=none,anchor=west] at (-6.5,0) {$\boldsymbol{\Sigma}$};
			\node[draw=none,fill=none,anchor=west] at (-7.5,-0.5) {\scriptsize$\boldsymbol{[t-1]}$};
			
			\draw[line width=1.05pt, rotate around={290:(0,0)}, dashed] (0,0) ellipse (2.625cm and 6.125cm);
			\node[draw=none,fill=none,anchor=west] at (5,2) {$\boldsymbol{\Sub}$};
			
			\draw[line width=1.05pt, rotate around={290:(-4.73,0.24)}, dashed, line join=round, line cap=round] (-4.73,0.24) arc (180:0:1.84cm and 2.625cm);
			\node[draw=none,fill=none,anchor=west] at (-2,0) {$\boldsymbol{\HyperSub}$};
			
			\draw[fill=black] (2,2.5) circle (3pt);
			\node[draw=none,fill=none,anchor=north west] at (2,2.5) {$\boldsymbol{P_1}$};
			
			\draw[fill=black] (0,0.5) circle (3pt);
			\node[draw=none,fill=none,anchor=north west] at (0,0.5) {$\boldsymbol{P_2}$};
			
			\draw[fill=black] (-3,-2.5) circle (3pt);
			\node[draw=none,fill=none,anchor=south east] at (-3,-2.5) {$\boldsymbol{Q}$};
			
			\draw[fill=black] (-4.5,-2.3) circle (3pt);
			\draw[fill=black] (-5,-1.2) circle (3pt);
			\draw[fill=black] (-3.8,-0.2) circle (3pt);
			\draw[fill=black] (-3,-1) circle (3pt);
			
			\draw[line width=1pt] (2.5,3) -- (-3,-2.5);
			
			
			\draw[->, line width=2pt] (6,-1) to [bend left] (6,-3);
			\node[draw=none,fill=none,anchor=west] at (6.4,-2) {$\boldsymbol{\FieldRed}$};
			
			
			\draw[line width=1.2pt, rotate around={290:(0,-7)}] (0,-7) ellipse (3cm and 7cm);
			\node[draw,fill=none,anchor=west] at (-7.6,-5) {\large\textbf{PG}$\boldsymbol{(st+s+t,q)}$};
			\node[draw,fill=none,anchor=west] at (-7.6,-5) {\large\textbf{PG}$\boldsymbol{(st+s+t,q)}$}; 
			
			\draw[line width=1.2pt, rotate around={290:(-4.5,-5.98)}, line join=round, line cap=round] (-4.5,-5.98) arc (180:0:2.5cm and 3cm);
			\node[draw=none,fill=none,anchor=west] at (-6.4,-6.4) {$\boldsymbol{\FieldRed(\Sigma)}$};
			\node[draw=none,fill=none,anchor=west] at (-7.4,-6.9) {\scriptsize$\boldsymbol{[st+t-1]}$};
			
			\draw[line width=1.05pt] (2,-4.5) circle (10pt);
			\node[draw=none,fill=none,anchor=north west] at (2.2,-4.7) {$\boldsymbol{\FieldRed(P_1)}$};
			\node[draw=none,fill=none] at (2,-4.5) {\scriptsize$\boldsymbol{[s]}$};
			
			\draw[line width=1.05pt] (0,-6.5) circle (10pt);
			\node[draw=none,fill=none,anchor=north west] at (0.2,-6.7) {$\boldsymbol{\FieldRed(P_2)}$};
			\node[draw=none,fill=none] at (0,-6.5) {\scriptsize$\boldsymbol{[s]}$};
			
			\draw[line width=1.05pt] (-3,-9.5) circle (10pt);
			\node[draw=none,fill=none,anchor=north west] at (-2.8,-9.7) {$\boldsymbol{\FieldRed(Q)}$};
			
			\draw[line width=1.05pt] (-4.5,-9.3) circle (10pt);
			\draw[line width=1.05pt] (-5,-8.2) circle (10pt);
			\draw[line width=1.05pt] (-3.8,-7.2) circle (10pt);
			\draw[line width=1.05pt] (-3,-8) circle (10pt);
			
			\draw[line width=1.05pt] (1.9,-3.9) -- (-3.6,-9.4) to [out=225, in=225](-2.9,-10.1) -- (2.6,-4.6) to [out=45, in=45] cycle;
			\node[draw=none,fill=none,anchor=east] at (0.7,-5) {$\boldsymbol{\Psi}$};
			\node[draw=none,fill=none,anchor=east] at (0.3,-5.3) {\scriptsize$\boldsymbol{[2s+1]}$};
			
			\draw[line width=1.05pt, densely dotted] (-3,-9.5) to [out=225, in=315] (-4.5,-9.3) to [out=135, in=260] (-5,-8.2) to [out=80, in=180] (-3.8,-7.2) to [out=0, in=100] (-3,-8) to [out=280, in=45] cycle;
			\node[draw=none,fill=none] at (-3.95,-8.2) {$\boldsymbol{\chi}$};
			\node[draw=none,fill=none] at (-3.8,-8.7) {\scriptsize$\boldsymbol{[st-1]}$};
			
			
			\draw[->, line width=2pt] (6,-8) to [bend left] (6,-10);
			\node[draw=none,fill=none,anchor=west] at (6.4,-9) {$\boldsymbol{d}$};
			
			
			\draw[line width=1.2pt, rotate around={290:(0,-14)}] (0,-14) ellipse (3cm and 7cm);
			\node[draw,fill=none,anchor=west] at (-7.6,-12) {\large\textbf{PG}$\boldsymbol{(st+s+t,q)}$};
			\node[draw,fill=none,anchor=west] at (-7.6,-12) {\large\textbf{PG}$\boldsymbol{(st+s+t,q)}$}; 
			
			\draw[line width=1.2pt, rotate around={290:(-6.43,-15.05)}, line join=round, line cap=round] (-6.43,-15.05) arc (210:-30:1.4cm and 1cm);
			\node[draw=none,fill=none,anchor=west] at (-8,-15.3) {$\boldsymbol{\FieldRed(\Sigma)^d}$};
			\node[draw=none,fill=none,rotate around={90:(0,0)}] at (-7.3,-15.85) {$\boldsymbol{:=}$};
			\node[draw=none,fill=none,anchor=west] at (-7.58,-16.3) {$\boldsymbol{\pi}$};
			\node[draw=none,fill=none,anchor=west] at (-6,-16) {\scriptsize$\boldsymbol{[s]}$};
			
			\draw[line width=1.2pt, rotate around={155:(0,-16.3)}, line join=round, line cap=round] (0,-16.3) arc (140:0:1.7cm and 2.6cm);
			\node[draw=none,fill=none,anchor=east] at (-1.3,-13.9) {$\boldsymbol{\Psi^d}$};
			\node[draw=none,fill=none,anchor=east] at (-1.06,-14.3) {\scriptsize$\boldsymbol{[st-s+t-2]}$};
			
			\draw[line width=1.2pt, rotate around={335:(-2.01,-13.51)}, line join=round, line cap=round] (-2.01,-13.51) arc (180:-12:1.678cm and 3.4cm);
			\node[draw=none,fill=none,anchor=west] at (1,-11) {$\boldsymbol{\FieldRed(P_1)^d}$};
			\node[draw=none,fill=none] at (2.2,-11.46) {\scriptsize$\boldsymbol{[st+t-1]}$};
			
			\draw[line width=1.2pt, rotate around={280:(-2,-13.53)}, line join=round, line cap=round] (-2,-13.53) arc (160:25:1.785cm and 9cm);
			\node[draw=none,fill=none,anchor=west] at (3.3,-13.44) {$\boldsymbol{\FieldRed(P_2)^d}$};
			\node[draw=none,fill=none] at (4.9,-13.9) {\scriptsize$\boldsymbol{[st+t-1]}$};
			
			\draw[line width=1.2pt, rotate around={280:(-6.32,-14.85)}] (-6.32,-14.85) arc (230:-40:1.9cm and 4cm);
			\node[draw=none,fill=none,anchor=north west] at (-5,-13.15) {$\boldsymbol{\FieldRed(Q)^d}$};
			\node[draw=none,fill=none] at (-3.6,-13.1) {\scriptsize$\boldsymbol{[st+t-1]}$};
			
			\draw[line width=1.05pt, densely dotted, line join=round] (-6.278,-14.833) to [out=35, in=225] (-1,-14.553) to [out=45, in=160] (0.4,-12) to [out=340, in=135] (-0.1,-14.668) to [out=315, in=100] (2.6,-14.2) to [out=280, in=30] (-0.4,-15.566) to [out=210, in=0] (-5.38,-17.404);
			\draw[line width=0.8pt, dotted] (-0.6,-15) circle (0.6cm);
			\node[draw=none,fill=none] at (-3.7,-15.1) {$\boldsymbol{\chi^d}$};
			\node[draw=none,fill=none] at (-3.6,-15.5) {\scriptsize$\boldsymbol{[s+t]}$};
			
			\node[draw=none,fill=none] at (0.2,-12.4) {\tiny$\boldsymbol{[t-1]}$};
			\node[draw=none,fill=none] at (2,-14.6) {\tiny$\boldsymbol{[t-1]}$};
			\node[draw=none,fill=none] at (-0.6,-15) {\tiny$\boldsymbol{[t-2]}$};
		\end{tikzpicture}\end{center}
		\caption{Visualisation of the map $\varphi$, see Definition \ref{Def_IsomorphismPhi}.}\label{Fig_PointLineIsomorphism}
		\end{figure}
	    
	    \bigskip
	    \underline{Claim $2$}: $\varphi$ maps points contained in a fixed element of $\llinesY$ onto points contained in a fixed element of $\llinesX$.
	    
	    \bigskip
	    At this point, Figure \ref{Fig_PointLineIsomorphism} comes in handy to visualise the following arguments.
	    
	    Let $\Sub\setminus\HyperSub\in\llinesY$ be arbitrary and let $P_1,P_2\in\Sub\setminus\HyperSub$ be two distinct points.
	    Define $Q:=P_1P_2\cap\HyperSub$.
	    By Lemma \ref{Lm_FieldRedProps}, $\FieldRed(P_1)$ and $\FieldRed(P_2)$ are disjoint $s$-spaces, each of which are disjoint to the $(st+t-1)$-space $\FieldRed(\Sigma)\supseteq\FieldRed(Q)$.
	    Therefore, they span a $(2s+1)$-space $\Psi:=\vspan{\FieldRed(P_1),\FieldRed(P_2)}$; note that any two distinct elements of the set $\{\FieldRed(P_1),\FieldRed(P_2),\FieldRed(Q)\}$ span $\Psi$.
	    Dualising these observations, we get that $\FieldRed(P_1)^d$, $\FieldRed(P_2)^d$ and $\FieldRed(Q)^d$ are $(st+t-1)$-spaces intersecting each other in the $(st-s+t-2)$-space $\Psi^d$, where $\FieldRed(P_1)^d$ and $\FieldRed(P_2)^d$ are disjoint to $\pi=\FieldRed(\Sigma)^d$; note that any two distinct elements of the set $\{\FieldRed(P_1)^d,\FieldRed(P_2)^d,\FieldRed(Q)^d\}$ intersect each other exactly in $\Psi^d$.
	    Moreover, the fact that $Q\in\Sigma$ means that $\FieldRed(Q)^d$ is an $(st+t-1)$-space going through $\pi$.
	    
	    As $\chi$ intersects $\FieldRed(Q)$ exactly in an $(s-1)$-space, $\chi^d$ intersects $\FieldRed(Q)^d$ in an $(s+t-1)$-space $\varphi(Q)=\FieldRed(Q)^d\cap\chi^d$.
	    Moreover, as both $\chi^d$ and $\FieldRed(Q)^d$ are spaces through $\pi$, their intersection $\varphi(Q)$ contains $\pi$ as well.
	    The space $\Psi^d$ is disjoint to $\pi$, so $\FieldRed(Q)^d$ is spanned by $\pi$ and $\Psi^d$; by Grassmann's identity, the $(s+t-1)$-space $\varphi(Q)=\FieldRed(Q)^d\cap\chi^d$ intersects $\Psi^d$ in a $(t-2)$-space.
	    This means that $\chi^d$ intersects both $\FieldRed(P_1)^d$ and $\FieldRed(P_2)^d$ in a $(t-2)$-space, which implies that $\varphi(P_1)$ intersects $\varphi(P_2)$ in a $(t-2)$-space.
	    
	    As $P_1$ and $P_2$ were arbitrarily chosen points of $\Sub\setminus\HyperSub$, we conclude that for any two points of the latter point set, their images under $\varphi$ intersect each other maximally.
	    Hence, this set of images $\varphi(\Sub\setminus\HyperSub)$ forms an Erd\H os-Ko-Rado set, which means that
	    \begin{enumerate}
	        \item either all elements of $\varphi(\Sub\setminus\HyperSub)$ lie in a $t$-space, or
	        \item all elements of $\varphi(\Sub\setminus\HyperSub)$ have a fixed $(t-2)$-space in common.
	    \end{enumerate}
	    If $1.$ generally holds, the proof of the claim is done, as this $t$-space is contained in $\chi^d$ and contains (at least) a $(t-1)$-space (of $\varphi(\Sub\setminus\HyperSub)$) disjoint to $\pi$.
	    Hence, by Grassmann's identity, this $t$-space intersects $\pi$ exactly in one point.
	    
	    Suppose that $2.$ holds.
	    Note that the points of the set $\Sub\setminus\HyperSub$ span the whole space $\pg\big(t,q^{s+1}\big)$.
	    This means that, by Lemma \ref{Lm_FieldRedProps}, all elements of $\FieldRed(\Sub\setminus\HyperSub)$ span the whole space $\pg(st+s+t,q)$.
	    However, as $2.$ holds, the intersection of all elements of $\varphi(\Sub\setminus\HyperSub)$ has dimension at least $t-2$, hence the intersection of all elements of $\FieldRed(\Sub\setminus\HyperSub)^d$ has dimension at least $t-2$ as well.
	    Dualising this statement, we obtain that the span of all elements of $\FieldRed(\Sub\setminus\HyperSub)$ has dimension at most $st+s+t-(t-2)-1=st+s+1$.
	    This is only possible if $st+s+t\leq st+s+1\Leftrightarrow t\leq1\Leftrightarrow t=1$.
	    
	    Hence, this implies that $t=1$.
	    Then $\chi^d$ is an $(s+1)$-space of $\pg(2s+1,q)$ through the $s$-space $\pi$, intersecting each element of $\varphi(\Sub\setminus\HyperSub)$ exactly in a point.
	    Denote the set of points in $\Sub$ by $\mathcal{P}_\Sub$.
	    As $\Sub\cong\pg(1,q)$, by Lemma \ref{Lm_SublineIsRegulus}, $\FieldRed(\mathcal{P}_\Sub)$ is a regulus of $\pg(2s+1,q)$.
	    Let $Q_1,Q_2\in\Sub\setminus\HyperSub$ and define $Q_1':=\varphi(Q_1)$ and $Q_2':=\varphi(Q_2)$.
	    Then $Q_1'Q_2'$ lies in the $(s+1)$-space $\chi^d$ and hence intersects the $s$-space $\pi=\FieldRed(Q)$ (here, $Q=\HyperSub$).
	    In this way, we see that $Q_1'Q_2'$ meets at least three elements of the regulus $\FieldRed(\mathcal{P}_\Sub)$ (namely $\FieldRed(Q_1)$, $\FieldRed(Q_2)$ and $\FieldRed(Q)$), hence $Q_1'Q_2'$ has to intersect all elements of that regulus.
	    As $Q_1'Q_2'$ is contained in $\chi^d$ and as each element of $\FieldRed(\Sub\setminus\HyperSub)$ intersects $\chi^d$ exactly in a point, all these intersection points have to lie on $Q_1'Q_2'$ and hence the proof is done.
	\end{proof}
	
	\bigskip
    \textbf{Acknowledgements.}
    The author would like to express a lot of gratitude towards Stefaan De Winter for his short and elegant idea on how to construct the isomorphism stated in Definition \ref{Def_IsomorphismPhi}.
    
    \bigskip
    \textbf{Conflict of interest.}
    On behalf of all authors, the corresponding author states that there is no conflict of interest.
    
    \bigskip
    \textbf{Data availability.}
    Data sharing not applicable to this article as no datasets were generated or analysed during the current study.
	
	\bibliographystyle{plain}
    \bibliography{main}

    \bigskip
    Author's address:
    
    \bigskip
    Lins Denaux
    
    Ghent University
    
    Department of Mathematics: Analysis, Logic and Discrete Mathematics
    
    Krijgslaan $281$ -- Building S$8$
    
    $9000$ Ghent
    
    BELGIUM
    
    \texttt{e-mail : lins.denaux@ugent.be}
    
    \texttt{website: }\url{https://users.ugent.be/~ldnaux}
    
\end{document}